\documentclass[a4paper,oneside,8pt]{article}%
\usepackage{makeidx}
\usepackage[english]{babel}
\usepackage{amsmath}
\usepackage{amsfonts}
\usepackage{amssymb}
\usepackage{stmaryrd}
\usepackage{graphicx}
\usepackage{mathrsfs}
\usepackage[colorlinks,linkcolor=red,anchorcolor=blue,citecolor=blue,urlcolor=blue]{hyperref}
\usepackage[symbol*,ragged]{footmisc}

\allowdisplaybreaks[4]
\providecommand{\U}[1]{\protect\rule{.1in}{.1in}}
\providecommand{\U}[1]{\protect \rule{.1in}{.1in}}

\newtheorem{theorem}{Theorem}[section]

\newtheorem{lemma}[theorem]{Lemma}

\newenvironment{proof}[1][Proof]{\noindent \textbf{#1.} }{\  \rule{0.5em}{0.5em}}
\numberwithin{equation}{section}

\begin{document}

\title{On two Diophantine inequalities over primes (II) }

\author{Yuetong Zhao\footnotemark[1] \,\,\,\,\,  \& \,\,\, Jinjiang Li\footnotemark[2]\,\,\,\,\,\,\,  \& \,\,\,
        Min Zhang\footnotemark[3]
                    \vspace*{-4mm} \\
     $\textrm{\small Department of Mathematics, China University of Mining and Technology\footnotemark[1]\,\,\,\footnotemark[2]}$
                    \vspace*{-4mm} \\
     \small  Beijing 100083, P. R. China
                     \vspace*{-4mm}  \\
     $\textrm{\small School of Applied Science, Beijing Information Science and Technology University\footnotemark[3]}$
                    \vspace*{-4mm}  \\
     \small  Beijing 100192, P. R. China  \vspace*{-4mm}  \\}

\footnotetext[2]{Corresponding author. \\
    \quad\,\, \textit{ E-mail addresses}:
     \href{mailto:yuetong.zhao.math@gmail.com}{yuetong.zhao.math@gmail.com} (Y. Zhao),
     \href{mailto:jinjiang.li.math@gmail.com}{jinjiang.li.math@gmail.com} (J. Li),\\
     \qquad \qquad\qquad\qquad\quad\quad \,\!\!
     \href{mailto:min.zhang.math@gmail.com}{min.zhang.math@gmail.com} (M. Zhang).  }

\date{}
\maketitle

{\textbf{Abstract}}: Let $1<c<\frac{26088036}{12301745},c\not=2$ and $N$ be a sufficiently large real number. In this paper, it is proved that, for almost all $R\in (N,2N]$, the Diophantine inequality
\begin{equation*}
    \big|p_1^c+p_2^c+p_3^c-R\big|<\log^{-1}N
\end{equation*}
is solvable in primes $p_1,p_2,p_3$. Moreover, we also prove that the following Diophantine inequality
\begin{equation*}
    \big|p_1^c+p_2^c+p_3^c+p_4^c+p_5^c+p_6^c-N\big|<\log^{-1}N
\end{equation*}
is solvable in prime variables $p_1,p_2,p_3,p_4,p_5,p_6$, which improves the previous result $1<c<\frac{37}{18},c\neq2$.


{\textbf{MR(2020) Subject Classification}}: 11J25, 11P32, 11P55, 11L07, 11L20

\section{Introduction and main result}
Let $k\geqslant1$ be a fixed integer and $N$ a sufficiently large integer. The famous Waring--Goldbach problem is to study the solvability of the following Diophantine equality
\begin{equation}\label{Waring-Goldbach-general}
  N=p_1^k+p_2^k+\cdots+p_r^k
\end{equation}
in prime variables $p_1,p_2,\dots,p_k$. For $k=2$, in 1938, Hua \cite{Hua-1938} proved that the equation (\ref{Waring-Goldbach-general}) is solvable for $r=5$ and sufficiently large integer $N$ satisfying $N\equiv 5\pmod {24}$.

In 1952, Piatetski--Shapiro \cite{Piatetski-Shapiro-1952} studied the following analogue of the Waring--Goldbach problem:
Suppose that $c>1$ is not an integer, $\varepsilon$ is a small positive number, and $N$ is a sufficiently large real number. Denote
by $H(c)$ the smallest natural number $r$ such that the following Diophantine inequality
\begin{equation}\label{Diophantine-inequality-general}
  |p_1^c+p_2^c+\cdots+p_r^c-N|<\varepsilon
\end{equation}
is solvable in primes $p_1,p_2,\dots,p_r$, then it was proved in \cite{Piatetski-Shapiro-1952} that
\begin{equation*}
  \limsup_{c\to+\infty}\frac{H(c)}{c\log c}\leqslant4.
\end{equation*}
In \cite{Piatetski-Shapiro-1952}, Piatetski--Shapiro considered the case $r=5$ in (\ref{Diophantine-inequality-general}) and proved that $H(c)\leqslant5$
for $1<c<3/2$. Later, the upper bound $3/2$ for $H(c)\leqslant5$ was improved successively to
\begin{equation*}
  \frac{14142}{8923}=1.584\cdots, \frac{1+\sqrt{5}}{2}=1.618\cdots, \frac{81}{40}=2.025, \frac{108}{53}=2.037\cdots,2.041,\cdots,\frac{52}{25}
\end{equation*}
by Zhai and Cao \cite{Zhai-Cao-2003}, Garaev \cite{Garaev-2003}, Zhai and Cao \cite{Zhai-Cao-2007}, Shi and Liu \cite{Shi-Liu-2013},
Baker and Weingartner \cite{Baker-Weingartner-2013}, Li and Cai \cite{Li-Cai-2020}, respectively. Especially, the results in \cite{Zhai-Cao-2007,Shi-Liu-2013,Baker-Weingartner-2013,Li-Cai-2020} satisfy $c>2$, which can be regarded as an analogue of Hua's theorem on sums of five squares of primes. By noting the fact that, for $c>2$, the sequence $p^c$ is sparser than the sequence $p^2$, thus the solvability of (\ref{Diophantine-inequality-general}) becomes more difficult when the range
of $c$, which satisfies $c>2$, becomes larger.

From these results and the Goldbach--Vinogradov theorem, it is reasonable to conjecture that if $c$ is
near to $1$, then the Diophantine inequality (\ref{Diophantine-inequality-general}) is solvable for $r=3$. This conjecture was first
established by Tolev \cite{Tolev-PhD-thesis} for $1<c<\frac{27}{26}$. Since then, the range of $c$ was enlarged to
\begin{equation*}
  \frac{15}{14},\quad \frac{13}{12},\quad \frac{11}{10},\quad \frac{237}{214},\quad \frac{61}{55},\quad \frac{10}{9},\quad \frac{43}{36}
\end{equation*}
by Tolev \cite{Tolev-1992}, Cai \cite{Cai-1996}, Cai \cite{Cai-1999} and Kumchev and Nedeva \cite{Kumchev-Nedeva-1998} independently, Cao and Zhai \cite{Cao-Zhai-2002},
Kumchev \cite{Kumchev-1999}, Baker and Weingartner \cite{Baker-Weingartner-2014}, Cai \cite{Cai-2018}, successively and respectively.

Combining Tolev's method and  the techniques of estimates on exponential sums of Fouvry and
Iwaniec \cite{Fouvry-Iwaniec-1989}, in 2003, Zhai and Cao \cite{Zhai-Cao-2003} proved that
$H(c)\leqslant4$ for $1<c<\frac{81}{68}$. Later, the range of $c$ for $H(c)\leqslant4$ was enlarged to $1<c<\frac{97}{81}$ by Mu \cite{Mu-Quanwu-2015}, to $1<c<\frac{6}{5}$ by Zhang and Li \cite{Zhang-Li-2019-JNT}, and to $1<c<\frac{1193}{889}$ by Zhang and Li \cite{Zhang-Li-2018-IJNT-4}, successively and respectively.

In 1999, Laporta \cite{Laporta-1999} studied the corresponding binary problem and proved that for $1<c<\frac{15}{14}$,  the Diophantine inequality
\begin{equation}\label{binary}
|p_1^c+p_2^c-R|<\varepsilon
\end{equation}
is solvable in primes $p_1,p_2$ for almost all $R\in(N,2N]$, where $N$ is a large real number and $\varepsilon=N^{1-\frac{15}{14c}}\log^8N$. Zhai and Cao \cite{Zhai-Cao-2003-2} and Kumchev and Laporta \cite{Kumchev-Laporta-2002} enlarged the range of $c$ to $1<c<\frac{43}{36}$ and $1<c<\frac{6}{5}$ respectively. Recently, the range of $c$ was further improved. Li and Cai \cite{Li-Cai-2020} proved that for $1<c<\frac{59}{44}$ with $c\neq\frac{4}{3}$, (\ref{binary}) is solvable for almost all $R\in(N,2N]$, and also (\ref{Diophantine-inequality-general}) is solvable for $r=4$.

In 2018, Zhang and Li \cite{Zhang-Li-2018} proved that for $1<c<\frac{37}{18},\, c\not=2$, for almost all $R\in(N,2N]$, the Diophantine inequality
\begin{equation*}
|p_1^c+p_2^c+p_3^c-R|<\log^{-1}N
\end{equation*}
is solvable in primes $p_1,p_2,p_3$, where $N$ is a large real number. Also, if $1<c<\frac{37}{18},\, c\not=2$, (\ref{Diophantine-inequality-general}) is solvable for $r=6$ .

 In this paper, we shall continue to improve the previous result in \cite{Zhang-Li-2018} and establish the following two theorems.

\begin{theorem}\label{Theorem-three-almost all}
   Let $1<c<\frac{26088036}{12301745},c\not=2$ and $N$ be a sufficiently large real number. Then for all $R\in (N,2N]\setminus \mathfrak{A}$ with
\begin{equation*}
|\mathfrak{A}|\ll N \exp\left(-\frac{2}{15}\bigg(\frac{1}{c}\log \frac{N}{3}\bigg)^{1/5}\right),
\end{equation*}
the Diophantine inequality
\begin{equation}\label{qu-3}
   \big|p_1^c+p_2^c+p_3^c-R\big|<\log^{-1}N
\end{equation}
is solvable in primes $p_1,p_2,p_3$.
\end{theorem}

\begin{theorem}\label{Theorem-five-variables-inequality}
   Suppose that $1<c<\frac{26088036}{12301745},c\not=2$, then for any sufficiently large real number $N$, the following Diophantine inequality
\begin{equation}\label{qu-5}
   \big|p_1^c+p_2^c+p_3^c+p_4^c+p_5^c+p_6^c-N\big|<\log^{-1}N
\end{equation}
is solvable in primes $p_1,p_2,p_3,p_4,p_5,p_6$.
\end{theorem}

\smallskip
\textbf{Remark 1.} In order to compare our result with the previous result, we list the numerical result as follows
\begin{equation*}
   \frac{26088036}{12301745}=2.120677676\cdots;\,\qquad \frac{37}{18}=2.0555555\cdots.
\end{equation*}

\smallskip
\textbf{Remark 2.} As is shown in \cite{Zhang-Li-2018}, it is proved that for $1<c<2$, (\ref{qu-5}) is solvable and (\ref{qu-3}) is solvable for almost all $R\in (N,2N]$. Therefore, in this paper, we only focus on the case $2<c<\frac{26088036}{12301745}$.

\smallskip
\textbf{Notation.}
Throughout this paper, we suppose that $1<c<\frac{26088036}{12301745},c\not=2$. Let $p$, with or without subscripts, always denote a prime number. $\eta$
always denotes an arbitrary small positive constant, which may not be the same at different occurrences; $N$ always denotes a sufficiently large real number.
As usual, we use $\Lambda(n)$ to denote von Mangoldt's function; $e(x)=e^{2\pi i x}$; $f(x)\ll g(x)$ means that $f(x)=O(g(x))$;
$f(x)\asymp g(x)$ means $f(x)\ll g(x)\ll f(x)$. Let $X$ be a parameter satisfying $X\asymp N^{\frac{1}{c}}$, which will be specified later. Also, We define
\begin{align*}
 & \,\,  \varepsilon=\log^{-4}X,\qquad K=\log^{10}X, \qquad \tau=X^{1-c-\eta}, \qquad \mathscr{L}=\log X,
        \qquad E=\exp(-\mathscr{L}^{\frac{1}{5}}),  \\
 & \,\, \mathcal{T}(x)=\sum_{X<n\leqslant2X}e(n^cx),\quad  S(x)=\sum_{X<p\leqslant2X}(\log p)e\big(p^{c}x\big), \quad  I(x)=\int_X^{2X}e(t^cx)\mathrm{d}t.
\end{align*}

\section{Preliminary Lemmas}
In this section, we shall give some preliminary lemmas, which are necessary in the proof of Theorem \ref{Theorem-three-almost all} and Theorem \ref{Theorem-five-variables-inequality}.

\begin{lemma}\label{xiaobei-lemma}
   Let $a,b$ be real numbers, $0<b<a/4,$ and let $r$ be a positive integer. Then there exists a function $\phi(y)$ which is $r$ times
   continuously differentiable and such that
  \begin{equation*}
    \left\{
      \begin{array}{cll}
          \phi(y)=1,    & &  \textrm{if \quad} |y|\leqslant a-b, \\
          0<\phi(y)<1,  & &  \textrm{if \quad} a-b<|y|< a+b, \\
          \phi(y)=0,    & &  \textrm{if \quad} |y|\geqslant a+b,
      \end{array}
    \right.
  \end{equation*}
  and its Fourier transform
   \begin{equation*}
      \Phi(x)=\int_{-\infty}^{+\infty} e(-xy)\phi(y)\mathrm{d}y
   \end{equation*}
   satisfies the inequality
   \begin{equation*}
      \left|\Phi(x)\right|\leqslant\min\left(2a,\frac{1}{\pi|x|},\frac{1}{\pi|x|}\left(\frac{r}{2\pi|x|b}\right)^r\right).
   \end{equation*}
\end{lemma}
\begin{proof}
 See Piatetski--Shapiro~\cite{Piatetski-Shapiro-1952} or Segal~\cite{Segal-1933-1}.  $\hfill$
\end{proof}

\begin{lemma}\label{Fouvry-Iwaniec-chafen}
Let $M,Q\geqslant1$ and $z_m$ be complex numbers. Then we have
\begin{equation*}
 \Bigg|\sum_{M<m\leqslant2M}z_m\Bigg|^2\leqslant\bigg(2+\frac{M}{Q}\bigg)\sum_{|q|<Q}\bigg(1-\frac{|q|}{Q}\bigg)
 \sum_{M<m+q,m-q\leqslant2M}z_{m+q}\overline{z_{m-q}}.
\end{equation*}
\end{lemma}
\begin{proof}
 See Lemma 2 of Fouvry and Iwaniec \cite{Fouvry-Iwaniec-1989}.  $\hfill$
\end{proof}

\begin{lemma}\label{yijie-exp-pair}
Suppose that $f(x):[a,b]\to\mathbb{R}$ has continuous derivatives of arbitrary order on $[a,b]$, where $1\leqslant a<b\leqslant2a$. Suppose further that
\begin{equation*}
 \big|f^{(j)}(x)\big|\asymp \lambda_1 a^{1-j},\qquad j\geqslant1, \qquad x\in[a,b].
\end{equation*}
Then for any exponential pair $(\kappa,\lambda)$, we have
\begin{equation*}
 \sum_{a<n\leqslant b}e(f(n))\ll \lambda_1^\kappa a^\lambda+\lambda_1^{-1}.
\end{equation*}
\end{lemma}
\begin{proof}
 See (3.3.4) of Graham and Kolesnik \cite{Graham-Kolesnik-book}.  $\hfill$
\end{proof}

\begin{lemma}\label{Robert-Sagos-lemma}
  Suppose $Y>1,\gamma>0, c>1, c\not\in\mathbb{Z}$. Let $\mathscr{A}(Y;c,\gamma)$ denote the number of solutions of the inequality
 \begin{eqnarray*}
   \big|n_1^c+n_2^c-n_3^c-n_4^c\big|<\gamma,\qquad Y<n_1,n_2,n_3,n_4\leqslant2Y,
 \end{eqnarray*}
 then there holds
 \begin{equation*}
    \mathscr{A}(Y;c,\gamma)\ll(\gamma Y^{4-c}+Y^2)Y^\eta.
 \end{equation*}
\end{lemma}
\begin{proof}
 See Theorem 2 of Robert and Sargos \cite{Robert-Sargos-2006}.  $\hfill$
\end{proof}

\begin{lemma}\label{Laporta-lemma-1}
We have
\begin{equation*}
A=\max_{R_2\in (N,2N]}\int_N^{2N}\bigg|\int_{\tau<|x|<K}e\big((R_1-R_2)x\big)\mathrm{d}x\bigg|\mathrm{d}R_1\ll \mathscr{L}.
\end{equation*}
\end{lemma}
\begin{proof}
 See Laporta \cite{Laporta-1999}, Lemma 1.  $\hfill$
\end{proof}

Let $\Omega_1$ and $\Omega_2$ be measurable subsets of $\mathbb{R}^n$. Let
\begin{equation*}
\lVert f \rVert_j =\bigg(\int_{\Omega_j}|f(y)|^2\mathrm{d}y\bigg)^{\frac{1}{2}}, \qquad \langle f,g \rangle_j = \int_{\Omega_j}f(y)\overline{g(y)}\mathrm{d}y,
\end{equation*}
be the usual norm and inner product in $L^2(\Omega_j,\mathbb{C})(j=1,2)$, respectively.

\begin{lemma}\label{Laporta-lemma-2}
Let $c\in L^2(\Omega_1,\mathbb{C}), \,\xi \in L^2(\Omega_2,\mathbb{C})$, and let $\omega$ be a measurable complex valued function on $\Omega_1\times\Omega_2$ such that
\begin{equation*}
\sup_{x\in\Omega_1}\int_{\Omega_2}|\omega(x,y)|\mathrm{d}y<+\infty, \qquad \sup_{y\in\Omega_2}\int_{\Omega_1}|\omega(x,y)|\mathrm{d}x<+\infty.
\end{equation*}
Then we have
\begin{equation*}
\bigg|\int_{\Omega_1}c(x)\langle \xi ,\omega(x,\cdot)\rangle_2 \mathrm{d}x\bigg|\leqslant \lVert \xi \rVert_2 \lVert c \rVert_1\bigg(\sup_{x'\in\Omega_1}\int_{\Omega_1}\big|\langle\omega(x,\cdot),\omega(x',\cdot)\rangle_2\big|
\mathrm{d}x\bigg)^{\frac{1}{2}}.
\end{equation*}
\end{lemma}
\begin{proof}
 See Laporta \cite{Laporta-1999}, Lemma 2.  $\hfill$
\end{proof}

\begin{lemma}\label{S(x)-I(x)-fourth-power-lemma}
  For $1<c<3,c\not=2$, we have
 \begin{eqnarray}\label{S(x)-fourth-power}
   \int_{-\tau}^{+\tau}\big|S(x)\big|^4\mathrm{d}x\ll X^{4-c}\log^5X,
 \end{eqnarray}
 \begin{eqnarray}\label{I(x)-fourth-power}
   \int_{-\tau}^{+\tau}\big|I(x)\big|^4\mathrm{d}x\ll X^{4-c}\log^5X.
 \end{eqnarray}
\end{lemma}
\begin{proof}
 We only prove (\ref{S(x)-fourth-power}), and (\ref{I(x)-fourth-power}) can be proved likewise. It is easy to see that
 \begin{align}\label{S(x)-fourth-upper}
        \int_{-\tau}^{+\tau}|S(x)|^4\mathrm{d}x
    = &\,\,  \sum_{X<p_1,\dots\,p_4\leqslant 2X}(\log p_1)\cdots(\log p_4)\int_{-\tau}^{+\tau} e\big((p_1^c+p_2^c-p_3^c-p_4^c)x\big) \mathrm{d}x
                    \nonumber   \\
  \ll & \,\, \sum_{X<p_1,\dots,p_4\leqslant 2X}(\log p_1)\cdots(\log p_4)\cdot\min\bigg(\tau,\frac{1}{|p_1^c+p_2^c-p_3^c-p_4^c|}\bigg)
                    \nonumber   \\
  \ll & \,\, \mathcal{U}\tau\log^4X+\mathcal{V}\log^4X,
\end{align}
where
\begin{equation*}
   \mathcal{U}=\sum_{\substack{X<n_1,\,n_2,\,n_3,\,n_4\leqslant 2X\\
            |n_1^c+n_2^c-n_3^c-n_4^c|\leqslant 1/\tau}}1\,\,,
      \qquad \mathcal{V}=\sum_{\substack{X<n_1,\,n_2,\,n_3,\,n_4\leqslant 2X\\
            |n_1^c+n_2^c-n_3^c-n_4^c|> 1/\tau}}
            \frac{1}{|n_1^c+n_2^c-n_3^c-n_4^c|}.
\end{equation*}
On one hand, we have
\begin{align*}
    \mathcal{U}  \ll & \,\, \sum_{X<n_1\leqslant 2X} \sum_{X<n_2\leqslant 2X} \sum_{X<n_3\leqslant 2X}
          \sum_{\substack{X<n_4\leqslant 2X \\ (n_1^c+n_2^c-n_3^c-1/\tau)^{1/c}\leqslant n_4\leqslant (n_1^c+n_2^c-n_3^c+1/\tau)^{1/c} \\
                           n_1^c+n_2^c-n_3^c\asymp X^c }} 1   \\
     \ll &  \sum_{\substack{X<n_1,\,n_2,\,n_3\leqslant 2X\\n_1^c+n_2^c-n_3^c\asymp X^c}}
            \left(1+(n_1^c+n_2^c-n_3^c+1/\tau)^{1/c}-(n_1^c+n_2^c-n_3^c-1/\tau)^{1/c}\right),
\end{align*}
and by the mean--value theorem we get
\begin{equation}
   \mathcal{U}\ll X^3+\frac{1}{\tau}X^{4-c}.
\end{equation}
On the other hand, we have $\mathcal{V}\leqslant\sum_{\ell}\mathcal{V_\ell}$, where
\begin{equation}
   \mathcal{V}_\ell=\sum_{\substack{X<n_1,\,n_2,\,n_3,\,n_4\leqslant 2X\\
                    \ell<|n_1^c+n_2^c-n_3^c-n_4^c|\leqslant2\ell}}
                     \frac{1}{|n_1^c+n_2^c-n_3^c-n_4^c|}
\end{equation}
and $\ell$ takes the values $2^k\tau^{-1},\,k=0,1,2,\cdots,$ with $\ell\ll X^c$. Then, we deduce that
\begin{eqnarray*}
  \mathcal{V}_\ell & \ll  & \frac{1}{\ell} \sum_{\substack{X<n_1,\,n_2,\,n_3,\,n_4\leqslant 2X \\
                  (n_1^c+n_2^c-n_3^c+\ell)^{1/c}\leqslant n_4\leqslant (n_1^c+n_2^c-n_3^c+2\ell)^{1/c}\\
                  n_1^c+n_2^c-n_3^c\asymp X^c }}1.   \\
\end{eqnarray*}
For $\ell\geqslant1/\tau$ and $X<n_1,\,n_2,\,n_3\,\leqslant 2X$~with~$n_1^c+n_2^c-n_3^c\asymp X^c$, it is easy to see that
\begin{equation*}
  \big(n_1^c+n_2^c-n_3^c+2\ell\big)^{1/c}-\big(n_1^c+n_2^c-n_3^c+\ell\big)^{1/c}>1.
\end{equation*}
Therefore, by the mean--value theorem, there holds
\begin{equation}\label{V_ell-upper}
  \mathcal{V}_\ell\ll\frac{1}{\ell}\sum_{\substack{X<n_1,\,n_2,\,n_3 \leqslant 2X\\n_1^c+n_2^c-n_3^c\asymp X^c }}
           \left(\big(n_1^c+n_2^c-n_3^c+2\ell\big)^{1/c}-\big(n_1^c+n_2^c-n_3^c+\ell\big)^{1/c}\right)\ll X^{4-c}.
\end{equation}
Thus, the conclusion (\ref{S(x)-fourth-power}) follows from (\ref{S(x)-fourth-upper})--(\ref{V_ell-upper}).     $\hfill$
\end{proof}

\begin{lemma}\label{S(x)=I(x)+jieyu}
   For $1<c<3,\,c\neq2,\,|x|\leqslant\tau,$ we have
  \begin{equation*}
     S(x)=I(x)+O\left(Xe^{-(\log X)^{1/5}}\right).
  \end{equation*}
\end{lemma}
\begin{proof}
    The proof of Lemma \ref{S(x)=I(x)+jieyu} is similar to that of Lemma 14 in  Tolev~\cite{Tolev-1992}.  $\hfill$
\end{proof}

\begin{lemma}\label{5-power-main-low-bound}
   For $1<c<3,c\neq2,$
   we have
     \begin{equation*}
        \int_{-\infty}^{+\infty} I^6(x)\Phi(x)e(-Nx)\mathrm{d}x\gg\varepsilon X^{6-c}.
     \end{equation*}
\end{lemma}
\begin{proof}
See Lemma 2.4 of Zhang and Li \cite{Zhang-Li-2018}.   $\hfill$
\end{proof}

\begin{lemma}\label{Graham-Kolesnik-bijiao}
  Suppose that
  \begin{equation*}
       L(Q)=\sum_{i=1}^nA_iQ^{a_i}+\sum_{j=1}^mB_jQ^{-b_j},
  \end{equation*}
  where $A_i,\,B_j,\,a_i\,\textrm{and}\,\,b_j$ are positive. Assume further that~$Q_1\leqslant Q_2$. Then there exists some
  $\mathscr{Q}$ with $\mathscr{Q}\in[Q_1,Q_2]$ such that
   \begin{equation*}
      L(\mathscr{Q}) \ll \sum_{i=1}^{n}A_iQ_1^{a_i}+\sum_{j=1}^{m}B_jQ_2^{-b_j}
                         +\sum_{i=1}^n\sum_{j=1}^{m}\big(A_i^{b_j}B_j^{a_i}\big)^{1/(a_i+b_j)},
   \end{equation*}
   where the implied constant depends only on $n$ and $m$.
\end{lemma}
\begin{proof}
 See Graham and Kolesnik \cite{Graham-Kolesnik-book}, Lemma 2.4.  $\hfill$
\end{proof}

\begin{lemma}\label{Baker-sanjiao-modified}
  Let $\alpha,\beta\in\mathbb{R},\,\alpha\not=0,1,2,\,\beta\not=0,1,2,3,|a_m|\ll1,|b_\ell|\ll1$. For $F\gg ML^2$, we have
 \begin{align*}
   & \,\,\,(ML)^{-\varepsilon}\cdot\sum_{M<m\leqslant2M}\sum_{L<\ell\leqslant2L}a_mb_\ell e\bigg(F\frac{m^\alpha \ell^\beta}{M^\alpha L^\beta}\bigg)
                          \nonumber  \\
   \ll_{\alpha,\beta,\varepsilon} &\,\, M^{\frac{7}{8}}L^{\frac{13}{16}}F^{\frac{1}{16}}+M^{\frac{515}{544}}L^{\frac{243}{272}}F^{\frac{1}{68}}
          +M^{\frac{34}{37}}L^{\frac{31}{37}}F^{\frac{3}{74}} +M^{\frac{363}{400}}L^{\frac{22}{25}}F^{\frac{3}{100}}
                          \nonumber  \\
 &  +M^{\frac{167}{176}}L^{\frac{303}{352}}F^{\frac{9}{352}}+M^{\frac{383}{416}}L^{\frac{23}{26}}F^{\frac{3}{104}}
       +M^{\frac{579}{640}}L^{\frac{37}{40}}F^{\frac{3}{160}}+M^{\frac{2269}{2368}}L^{\frac{33}{37}}F^{\frac{9}{592}}
                          \nonumber  \\
&  +M^{\frac{711}{768}}L^{\frac{181}{192}}F^{\frac{1}{96}}+M^{\frac{93}{104}}L^{\frac{23}{26}}F^{\frac{1}{26}}
       +M^{\frac{8}{9}}L^{\frac{11}{12}}F^{\frac{1}{36}} +M^{\frac{479}{512}}L^{\frac{57}{64}}F^{\frac{3}{128}}
                           \nonumber  \\
&  +M^{\frac{61}{64}}L^{\frac{9}{8}}F^{-\frac{1}{16}} +M^{\frac{431}{448}}L^{\frac{27}{28}}F^{-\frac{1}{112}}
       +M^{\frac{101}{100}}L^{\frac{183}{200}}F^{-\frac{1}{100}} +M^{\frac{467}{512}}L^{\frac{65}{64}}F^{-\frac{1}{128}}
                            \nonumber  \\
&  +M^{\frac{61}{64}}L^{\frac{15}{16}}+M^{\frac{63}{64}}L^{\frac{29}{32}}+ML^{\frac{93}{104}}+M^{\frac{65}{72}}L+M^{\frac{235}{256}}L^{\frac{63}{64}}.
 \end{align*}
\end{lemma}
\begin{proof}
We follow the process of the proof of Theorem 1 of Baker and Weingartner \cite{Baker-Weingartner-2013} step by
step until p. 267 in \cite{Baker-Weingartner-2013}
and get
\begin{align*}
  & \,\,(ML)^{-\varepsilon}\bigg|\sum_{M<m\leqslant2M}\sum_{L<\ell\leqslant2L}a_mb_\ell e\bigg(F\frac{m^\alpha\ell^\beta}{M^\alpha L^\beta}\bigg)\bigg|^{16}
                            \nonumber  \\
  \ll & \,\, M^{14}L^{13}F+M^{14}L^{12}Q^{\frac{13}{3}}F+M^{\frac{53}{4}}L^{12}Q^{\frac{28}{3}}F+M^{\frac{53}{4}}L^{13}Q^5F+M^{16}L^{14}Q^{\frac{4}{3}}
                             \nonumber  \\
   & \,\,+M^{\frac{57}{4}}L^{16}Q+M^{17}L^{18}F^{-1}Q^{-7}+M^{16}L^{16}Q^{-8}+M^{15}L^{16}Q^{-4}+M^{16}L^{15}Q^{-3}.
\end{align*}
Next, according to the arguments on p. 267 of  Baker and Weingartner \cite{Baker-Weingartner-2013}, we use Lemma \ref{Graham-Kolesnik-bijiao} to optimize
$Q$ over $[1,M^{1/4}]$ and deduce that
\begin{align*}
  & \,\,(ML)^{-\varepsilon}\bigg|\sum_{M<m\leqslant2M}\sum_{L<\ell\leqslant2L}a_mb_\ell e\bigg(F\frac{m^\alpha\ell^\beta}{M^\alpha L^\beta}\bigg)\bigg|^{16}
                            \nonumber  \\
  \ll & \,\, M^{14}L^{13}F+M^{\frac{515}{34}}L^{\frac{243}{17}}F^{\frac{4}{17}}+M^{\frac{544}{37}}L^{\frac{496}{37}}F^{\frac{24}{37}}
             +M^{\frac{363}{25}}L^{\frac{352}{25}}F^{\frac{12}{25}}
                            \nonumber  \\
& \,\, +M^{\frac{167}{11}}L^{\frac{303}{22}}F^{\frac{9}{22}}+M^{\frac{383}{26}}L^{\frac{184}{13}}F^{\frac{6}{13}}
             +M^{\frac{579}{40}}L^{\frac{74}{5}}F^{\frac{3}{10}}+M^{\frac{2269}{148}}L^{\frac{528}{37}}F^{\frac{9}{37}}
                            \nonumber  \\
& \,\, +M^{\frac{711}{48}}L^{\frac{181}{12}}F^{\frac{1}{6}}+M^{\frac{186}{13}}L^{\frac{184}{13}}F^{\frac{8}{13}}
             +M^{\frac{128}{9}}L^{\frac{44}{3}}F^{\frac{4}{9}}+M^{\frac{479}{32}}L^{\frac{57}{4}}F^{\frac{3}{8}}
                             \nonumber  \\
& \,\, +M^{\frac{61}{4}}L^{18}F^{-1}+M^{\frac{431}{28}}L^{\frac{108}{7}}F^{-\frac{1}{7}}+M^{\frac{404}{25}}L^{\frac{366}{25}}F^{-\frac{4}{25}}
           +M^{\frac{467}{32}}L^{\frac{65}{4}}F^{-\frac{1}{8}}
                             \nonumber  \\
& \,\, +M^{\frac{61}{4}}L^{15}+M^{\frac{63}{4}}L^{\frac{29}{2}}+M^{16}L^{\frac{186}{13}}+M^{\frac{130}{9}}L^{16}+M^{\frac{235}{16}}L^{\frac{63}{4}},
\end{align*}
which implies the desired result.    $\hfill$
\end{proof}

\begin{lemma}\label{Heath-Brown-exponent-sum-fenjie}
Let $3<U<V<Z<X$ and suppose that $Z-\frac{1}{2}\in\mathbb{N},\,X\gg Z^2U,\,Z\gg U^2,\,V^3\gg X$. Assume further that $f(n)$ is a
complex--valued function such that $|f(n)|\leqslant1$. Then the sum
\begin{equation*}
 \sum_{X<n\leqslant2X}\Lambda(n)f(n)
\end{equation*}
can be decomposed into $O(\log^{10}X)$ sums, each of which either of Type I:
\begin{equation*}
 \sum_{M<m\leqslant2M}a(m)\sum_{L<\ell\leqslant2L}f(m\ell)
\end{equation*}
with $L\gg Z$, where $a(m)\ll m^{\eta},\,ML\asymp X$, or of Type II:
\begin{equation*}
 \sum_{M<m\leqslant2M}a(m)\sum_{L<\ell\leqslant2L}b(\ell)f(m\ell)
\end{equation*}
with $U\ll M\ll V$, where $a(m)\ll m^{\eta},\,b(\ell)\ll \ell^{\eta},\,ML\asymp X$.
\end{lemma}
\begin{proof}
 See Lemma 3 of Heath--Brown \cite{Heath-Brown-1983}.  $\hfill$
\end{proof}

In the rest of this paper, we always suppose that $2<c<\frac{26088036}{12301745}$, and set $F=|x|X^c$ for $\tau<|x|<K$. Trivially, we have $X^{1-\eta}\ll F\ll KX^c$.

\begin{lemma}\label{5-variables-Type-I-1}
Suppose that $\tau<|x|<K,\, F\ll ML^2,\, L\gg X^{\frac{12089667}{24603490}},a(m)\ll m^\eta,ML\asymp X$, then we have
\begin{equation*}
 S_I(M,L):=\sum_{M<m\leqslant2M}\sum_{L<\ell\leqslant2L}a(m)e(xm^c\ell^c)\ll X^{\frac{12195706}{12301745}+\eta}.
\end{equation*}
\end{lemma}
\begin{proof}
By Lemma \ref{yijie-exp-pair} with the exponential pair $(\kappa,\lambda)=A^2B(0,1)=(\frac{1}{14},\frac{11}{14})$, we deduce that
\begin{align*}
  S_I(M,L) \ll & \,\, X^\eta\sum_{M<m\leqslant2M}\Bigg|\sum_{L<\ell\leqslant2L}e(xm^c\ell^c)\Bigg|
                               \nonumber  \\
    \ll & \,\, X^\eta\sum_{M<m\leqslant2M}\bigg(\big(|x|X^cL^{-1}\big)^\frac{1}{14}L^{\frac{11}{14}}+\frac{1}{|x|X^cL^{-1}}\bigg)
                               \nonumber  \\
     \ll & \,\, X^\eta\Big(M\big(FL^{-1}\big)^{\frac{1}{14}}L^{\frac{11}{14}}+\tau^{-1}X^{1-c}\Big)
                                \nonumber  \\
     \ll & \,\, MF^{\frac{1}{14}}L^{\frac{5}{7}}X^\eta \ll M(ML^2)^{\frac{1}{14}}L^{\frac{5}{7}}X^\eta
                                 \nonumber  \\
     \ll & \,\, M^{\frac{15}{14}}L^{\frac{6}{7}}X^\eta\ll X^{\frac{15}{14}+\eta}L^{-\frac{3}{14}}
                \ll X^{\frac{12195706}{12301745}+\eta},
\end{align*}
which completes the proof of Lemma \ref{5-variables-Type-I-1}.      $\hfill$
\end{proof}

\begin{lemma}\label{5-variables-Type-I-2-1}
Suppose that $\tau<|x|<K,F\gg ML^2, M\ll X^{\frac{12513823}{24603490}}, a(m)\ll m^\eta,ML\asymp X$. Then we have
\begin{equation*}
 S_I(M,L):=\sum_{M<m\leqslant2M}\sum_{L<\ell\leqslant2L}a(m)e(xm^c\ell^c)\ll X^{\frac{12195706}{12301745}+\eta}.
\end{equation*}
\end{lemma}
\begin{proof}
For $M\ll X^{\frac{3393655}{12301745}}$, by Lemma \ref{yijie-exp-pair} with the exponential pair $(\kappa,\lambda)=AB(0,1)=(\frac{1}{6},\frac{2}{3})$,
we derive that
\begin{align*}
  S_I(M,L) \ll & \,\, M^\eta\sum_{M<m\leqslant2M}\Bigg|\sum_{L<\ell\leqslant2L}e(xm^c\ell^c)\Bigg|
                               \nonumber  \\
    \ll & \,\, M^\eta\sum_{M<m\leqslant2M}\bigg(\big(|x|X^cL^{-1}\big)^\frac{1}{6}L^{\frac{2}{3}}+\frac{1}{|x|X^cL^{-1}}\bigg)
                               \nonumber  \\
     \ll & \,\, X^\eta\Big(K^{\frac{1}{6}}X^{\frac{c}{6}}L^{\frac{1}{2}}M+\tau^{-1}X^{1-c}\Big)
                                \nonumber  \\
     \ll & \,\, M^{\frac{1}{2}}X^{\frac{c}{6}+\frac{1}{2}+\eta}\ll X^{\frac{12195706}{12301745}+\eta},
\end{align*}
which completes the proof of Lemma \ref{5-variables-Type-I-2-1}. If
$X^{\frac{3393655}{12301745}}\ll M\ll X^{\frac{12513823}{24603490}}$, it follows from Lemma \ref{Baker-sanjiao-modified} with $(m,\ell)=(m,\ell)$ that $ S_I(M,L)\ll X^{\frac{12195706}{12301745}+\eta}$,
which derives the desired result.  $\hfill$
\end{proof}

\begin{lemma}\label{5-variables-Type-II}
Suppose that  $\tau<|x|<K,X^{\frac{212078}{12301745}}\ll M\ll X^{\frac{28846271}{49206980}}, a(m)\ll m^\eta,b(\ell)\ll \ell^\eta, ML\asymp X$, then we have
\begin{equation*}
 \mathcal{S}_{II}(M,L):=\sum_{M<m\leqslant2M}\sum_{L<\ell\leqslant2L}a(m)b(\ell)e(xm^c\ell^c)\ll X^{\frac{12195706}{12301745}+\eta}.
\end{equation*}
\end{lemma}
\begin{proof}
Taking $Q=X^{\frac{212078}{12301745}}(\log X)^{-1}$,
if $X^{\frac{212078}{12301745}}\ll M\ll X^{\frac{28846271}{49206980}}$, by
Cauchy's inequality and Lemma \ref{Fouvry-Iwaniec-chafen}, we deduce that
\begin{align}\label{Type-II-inner-sum}
        & \,\,  S_{II}(M,L) \ll \Bigg|\sum_{L<\ell\leqslant2L}b(\ell)\sum_{M<m\leqslant2M}a(m)e(xm^c\ell^c)\Bigg|
                     \nonumber  \\
    \ll & \,\, \Bigg(\sum_{L<\ell\leqslant2L}|b(\ell)|^2\Bigg)^{\frac{1}{2}}
               \Bigg(\sum_{L<\ell\leqslant2L}\Bigg|\sum_{M<m\leqslant2M}a(m)e(xm^c\ell^c)\Bigg|^2\Bigg)^{\frac{1}{2}}
                     \nonumber  \\
    \ll & \,\, L^{\frac{1}{2}+\eta} \Bigg(\sum_{L<\ell\leqslant2L}\frac{M}{Q}\sum_{0\leqslant q<Q}\bigg(1-\frac{q}{Q}\bigg)
                     \nonumber  \\
    & \qquad \qquad \times\sum_{M+q<m\leqslant2M-q}a(m+q)\overline{a(m-q)}e\Big(x\ell^c\big((m+q)^c-(m-q)^c\big)\Big)\Bigg)^{\frac{1}{2}}
                      \nonumber  \\
    \ll & \,\, L^{\frac{1}{2}+\eta}\Bigg(\frac{M}{Q}\sum_{L<\ell\leqslant2L}\bigg(M^{1+\eta}+\sum_{1\leqslant q<Q}\bigg(1-\frac{q}{Q}\bigg)
                      \nonumber  \\
    & \qquad \qquad \times\sum_{M+q<m\leqslant2M-q}a(m+q)\overline{a(m-q)}e\Big(x\ell^c\big((m+q)^c-(m-q)^c\big)\Big) \bigg)\Bigg)^{\frac{1}{2}}
                      \nonumber  \\
    \ll & \,\, X^\eta\Bigg(\frac{X^2}{Q}+\frac{X}{Q}\sum_{1\leqslant q<Q}\sum_{M<m\leqslant2M}
               \Bigg|\sum_{L<\ell\leqslant2L}e\Big(x\ell^c\big((m+q)^c-(m-q)^c\big)\Big) \Bigg|\Bigg)^{\frac{1}{2}}.
\end{align}
Therefore, it suffices to give the estimate of the following sum
\begin{equation*}
  \mathscr{S}_0:=\sum_{L<\ell\leqslant2L}e\Big(x\ell^c\big((m+q)^c-(m-q)^c\big)\Big).
\end{equation*}
From Lemma \ref{yijie-exp-pair} with the exponential pair $(\kappa,\lambda)=A^3B(0,1)=(\frac{1}{30},\frac{26}{30})$, we have
\begin{equation*}
   \mathscr{S}_0\ll \big(|x|X^{c-1}q\big)^{\frac{1}{30}}L^{\frac{26}{30}}+\frac{1}{|x|X^{c-1}q}.
\end{equation*}
Inserting the above estimate into (\ref{Type-II-inner-sum}), we derive that
\begin{align*}
             S_{II}(M,L)
  \ll & \,\, X^\eta\Bigg(\frac{X^2}{Q}+\frac{X}{Q}\sum_{1\leqslant q<Q}
            \sum_{M<m\leqslant2M}\bigg(\big(|x|X^{c-1}q\big)^{\frac{1}{30}}L^{\frac{26}{30}}+\frac{1}{|x|X^{c-1}q}\bigg)\Bigg)^{\frac{1}{2}}
                           \nonumber  \\
  \ll & \,\, X^\eta\Bigg(\frac{X^2}{Q}+\frac{X}{Q}\Big(K^\frac{1}{30}X^{\frac{c-1}{30}}L^{\frac{26}{30}}MQ^{\frac{31}{30}}
             +\tau^{-1}X^{1-c}M\log Q\Big)\Bigg)^{\frac{1}{2}}
                             \nonumber  \\
   \ll & \,\, \big(X^{2+\eta}Q^{-1}\big)^{\frac{1}{2}}\ll X^{\frac{12195706}{12301745}+\eta},
\end{align*}
which completes the proof of Lemma \ref{5-variables-Type-II}.   $\hfill$
\end{proof}

\begin{lemma}\label{S(x)-yuqujianguji-5}
Suppose that $2<c<\frac{26088036}{12301745}$. For $\tau<|x|<K$, there holds
\begin{equation*}
 S(x)\ll X^{\frac{12195706}{12301745}+\eta}.
\end{equation*}
\end{lemma}
\begin{proof}
First, we have
\begin{equation}\label{S=U+error-5}
  S(x)=\mathscr{U}(x)+O(X^{1/2}),
\end{equation}
where
\begin{equation*}
  \mathscr{U}(x)=\sum_{X<n\leqslant2X}\Lambda(n)e(n^cx).
\end{equation*}
Taking $U=X^{\frac{212078}{12301745}},V=X^{\frac{28846271}{49206980}},Z=\big[X^{\frac{12089667}{24603490}}\big]+\frac{1}{2}$ in Lemma \ref{Heath-Brown-exponent-sum-fenjie},
it is easy to see that the sum
\begin{equation*}
 \sum_{X<n\leqslant2X}\Lambda(n)e(n^cx)
\end{equation*}
can be decomposed into $O(\log^{10}X)$ sums, each of which either of Type I:
\begin{equation*}
  S_{I}(M,L)=\sum_{M<m\leqslant2M}\sum_{L<\ell\leqslant2L}a(m)e(xm^c\ell^c)
\end{equation*}
with $L\gg Z, a(m)\ll m^\eta, ML\asymp X$, or of Type II:
\begin{equation*}
  S_{II}(M,L)=\sum_{M<m\leqslant2M}\sum_{L<\ell\leqslant2L}a(m)b(\ell)e(xm^c\ell^c)
\end{equation*}
with $U\ll M\ll V,a(m)\ll m^\eta,b(\ell)\ll \ell^\eta, ML\asymp X$. For the Type I sums, if $F\ll ML^2$, then from Lemma \ref{5-variables-Type-I-1}, we have $S_{I}(M,L)\ll X^{\frac{12195706}{12301745}+\eta}$;
if $F\gg ML^2$, then from Lemma \ref{5-variables-Type-I-2-1},
we have $S_{I}(M,L)\ll X^{\frac{12195706}{12301745}+\eta}$. For the Type II sums, from Lemma \ref{5-variables-Type-II},
we get  $S_{II}(M,L)\ll X^{\frac{12195706}{12301745}+\eta}$. Thus, we derive that
\begin{equation}\label{U-guji-5}
 \sum_{X<n\leqslant2X}\Lambda(n)e(n^cx)\ll X^{\frac{12195706}{12301745}+\eta}.
\end{equation}
By (\ref{S=U+error-5}) and (\ref{U-guji-5}), we complete the proof of Lemma \ref{S(x)-yuqujianguji-5}.  $\hfill$
\end{proof}

\section{Proof of Theorem \ref{Theorem-three-almost all} }
Throughout the proof of Theorem \ref{Theorem-three-almost all}, we set $X=(\frac{N}{3})^{\frac{1}{c}}$ and denote the function $\phi(y)$ and $\Phi(x)$ which are from Lemma \ref{xiaobei-lemma} with parameters $a=\frac{9\varepsilon}{10},b=\frac{\varepsilon}{10},r=[\log X]$.\\

In order to prove this theorem, first of all, we need to estimate the following integral
\begin{equation*}
  \int_\tau^K |S(x)|^4|\Phi(x)|\mathrm{d}x
\end{equation*}
under the condition $c>2$. We have
\begin{align}\label{S^4-Phi-1}
  & \,\, \int_\tau^K |S(x)|^4|\Phi(x)|\mathrm{d}x\ll \varepsilon \int_\tau^K |S(x)|^4\mathrm{d}x
                  \nonumber \\
  = & \,\, \varepsilon\sum_{X<p_1,\dots,p_4\leqslant2X}(\log p_1)\cdots(\log p_4)\int_\tau^Ke\big((p_1^c+p_2^c-p_3^c-p_4^c)x\big)\mathrm{d}x
                  \nonumber \\
  \ll & \,\, \varepsilon(\log X)^4\sum_{X<p_1,\dots,p_4\leqslant2X}\min\bigg(K,\frac{1}{|p_1^c+p_2^c-p_3^c-p_4^c|}\bigg)
                   \nonumber \\
  \ll & \,\, \sum_{X<n_1,\dots,n_4\leqslant2X}\min\bigg(K,\frac{1}{|n_1^c+n_2^c-n_3^c-n_4^c|}\bigg).
\end{align}
Set $u=n_1^c+n_2^c-n_3^c-n_4^c$, then by Lemma \ref{Robert-Sagos-lemma}, we know that the contribution of $K$ to (\ref{S^4-Phi-1}) is
\begin{align*}
  \ll & \,\, K\cdot\mathscr{A}(X;c,K^{-1})\ll K\big(K^{-1}X^{4-c}+X^2\big)X^{\eta} \\
  \ll & \,\,  \big(X^{4-c}+KX^2\big)X^{\eta} \ll X^{2+\eta}.
\end{align*}
By a splitting argument, the  contribution of $u$ with $|u|>K^{-1}$ to (\ref{S^4-Phi-1}) is
\begin{align*}
  \ll & \,\, (\log X)\max_{K^{-1}\ll U\ll X^c}\sum_{\substack{X<n_1,n_2,n_3,n_4\leqslant2X\\ U<|u|\leqslant2U}}\frac{1}{|u|}
                           \nonumber \\
  \ll & \,\, (\log X)\max_{K^{-1}\ll U\ll X^c} U^{-1}\cdot\mathscr{A}(X;c,2U)
                            \nonumber \\
  \ll & \,\,  (\log X)\max_{K^{-1}\ll U\ll X^c}U^{-1}\big(UX^{4-c}+X^2\big)X^{\eta}
                            \nonumber \\
  \ll & \,\,  \big(X^{4-c}+KX^2\big)X^{\eta} \ll X^{2+\eta}.
\end{align*}
Combining the above two cases, we deduce that, for $c>2$, there holds
\begin{equation}\label{S^4-Phi-upper}
  \int_\tau^K |S(x)|^4|\Phi(x)|\mathrm{d}x\ll X^{2+\eta}.
\end{equation}

Then, we consider the following integral
\begin{equation*}
\int_{\tau<|x|<K} |S(x)|^5|\Phi(x)|\mathrm{d}x.
\end{equation*}
We have
\begin{align*}
\int_{\tau<|x|<K} |S(x)|^5|\Phi(x)|\mathrm{d}x
= & \,\, \sum_{X<p\leqslant2X}(\log p)\int_{\tau<|x|<K}e(p^cx)\overline{S(x)}|S(x)|^3|\Phi(x)|\mathrm{d}x\\
\ll & \,\, (\log X)\sum_{X<n\leqslant2X}\Bigg|\int_{\tau<|x|<K}e(n^cx)\overline{S(x)}|S(x)|^3|\Phi(x)|\mathrm{d}x\Bigg|.
\end{align*}
From Cauchy's inequality, we obtain
\begin{align}\label{S^5-Cauchy-upper}
& \,\, \int_{\tau<|x|<K} |S(x)|^5|\Phi(x)|\mathrm{d}x\nonumber\\
\leqslant & \,\, X^{\frac{1}{2}}(\log X)\Bigg(\sum_{X<n\leqslant2X}\Bigg|\int_{\tau<|x|<K}e(n^cx)\overline{S(x)}|S(x)|^3|\Phi(x)|\mathrm{d}x\Bigg|^2\Bigg)^{\frac{1}{2}}\nonumber\\
= & \,\, X^{\frac{1}{2}}(\log X)\Bigg(\sum_{X<n\leqslant2X}\int_{\tau<|y|<K}\overline{e(n^cy)\overline{S(y)}|S(y)|^3|\Phi(y)|}\mathrm{d}y\nonumber\\
  & \,\, \qquad \qquad\qquad\qquad\qquad\qquad\times\int_{\tau<|x|<K}e(n^cx)\overline{S(x)}|S(x)|^3|\Phi(x)|\mathrm{d}x\Bigg)^{\frac{1}{2}}\nonumber\\
\leqslant & \,\, X^{\frac{1}{2}}(\log X)\Bigg(\int_{\tau<|y|<K}|S(y)|^4|\Phi(y)|\mathrm{d}y\int_{\tau<|x|<K}|S(x)|^4|\Phi(x)||\mathcal{T}(x-y)|\mathrm{d}x\Bigg)^{\frac{1}{2}}.
\end{align}
For the inner integral in (\ref{S^5-Cauchy-upper}), we have
\begin{align}\label{4-inner-integral-fenjie}
& \,\, \int_{\tau<|x|<K}|S(x)|^4|\Phi(x)||\mathcal{T}(x-y)|\mathrm{d}x\nonumber\\
\ll & \,\, \int_{\substack{\tau<|x|<K \\ |x-y|\leqslant X^{-c}}}|S(x)|^4|\Phi(x)||\mathcal{T}(x-y)|\mathrm{d}x+\int_{\substack{\tau<|x|<K \\ X^{-c}<|x-y|\leqslant 2K}}|S(x)|^4|\Phi(x)||\mathcal{T}(x-y)|\mathrm{d}x.
\end{align}
According to Lemma \ref{xiaobei-lemma}, Lemma \ref{S(x)-yuqujianguji-5} and the trivial estimate $\mathcal{T}(x-y)\ll X$, we get
\begin{align}\label{4-inner-integral-fenjie-1}
& \,\,\int_{\substack{\tau<|x|<K \\ |x-y|\leqslant X^{-c}}}|S(x)|^4|\Phi(x)||\mathcal{T}(x-y)|\mathrm{d}x \nonumber\\
\ll & \,\, \varepsilon X\times \sup_{\tau<|x|<K} |S(x)|^4\times
              \int_{\substack{\tau<|x|<K \\ |x-y|\leqslant X^{-c}}} \mathrm{d}x \nonumber \\
\ll & \,\, \varepsilon X\cdot X^{\frac{48782824}{12301745}-c}\ll \varepsilon X^{\frac{61084569}{12301745}-c}.
\end{align}
From Lemma \ref{yijie-exp-pair}, for $X^{-c}<|x|\leqslant2K$, we have
\begin{align}\label{T(x)-expo-pair-abs-5}
 \mathcal{T}(x) \ll & \,\, \big(|x|X^{c-1}\big)^\kappa X^\lambda+\frac{1}{|x|X^{c-1}}\nonumber \\
   \ll & \,\, |x|^\kappa X^{\kappa c+\lambda-\kappa}+\frac{1}{|x|X^{c-1}} .
\end{align}
Taking
\begin{align*}
  (\kappa,\lambda)= A BA^2    BA BA BA BA BA BA   BA^2 BA^2 BA^2 BA^2  B(0,1) = \bigg(\frac{156989}{1244758},\frac{875691}{1244758}\bigg)
\end{align*}
in (\ref{T(x)-expo-pair-abs-5}), we derive that, for $X^{-c}<|x-y|\leqslant2K$, there holds
\begin{align}\label{T(x)-expo-pair-explicit-4}
 \mathcal{T}(x-y)
   \ll & \,\, |x-y|^\frac{156989}{1244758} X^{\frac{156989}{1244758} c+\frac{359351}{622379}}
               +\frac{1}{|x-y|X^{c-1}}.
\end{align}
From (\ref{S^4-Phi-upper}), (\ref{T(x)-expo-pair-explicit-4}) and
Lemma \ref{S(x)-yuqujianguji-5}, we derive that
\begin{align}\label{4-inner-integral-fenjie-2}
 & \,\, \int_{\substack{\tau<|x|<K \\ X^{-c}<|x-y|\leqslant2K}}|S(x)|^4|\Phi(x)||\mathcal{T}(x-y)|\mathrm{d}x
                     \nonumber \\
 \ll & \,\, \int_{\substack{\tau<|x|<K \\ X^{-c}<|x-y|\leqslant2K}}|S(x)|^4|\Phi(x)|
              \bigg(|x-y|^\frac{156989}{1244758} X^{\frac{156989}{1244758} c+\frac{359351}{622379}}
               +\frac{1}{|x-y|X^{c-1}} \bigg)\mathrm{d}x
                     \nonumber \\
 \ll & \,\, X^{\frac{156989}{1244758} c+\frac{359351}{622379}+\eta}
            \int_{\tau<|x|<K}\!\!|S(x)|^4|\Phi(x)|\mathrm{d}x
                      \nonumber \\
    & \,\,  \qquad+\varepsilon X^{1-c}\times \sup_{\tau<|x|<K}|S(x)|^4\times \int_{\substack{\tau<|x|<K \\ X^{-c}<|x-y|\leqslant2K}}\frac{\mathrm{d}x}{|x-y|}
                        \nonumber \\
  \ll & \,\, X^{\frac{156989}{1244758} c+\frac{359351}{622379}+\eta}\cdot X^{2+\eta}
             +\varepsilon X^{1-c}\cdot X^{\frac{48782824}{12301745}+\eta}
                        \nonumber \\
 \ll & \,\, X^{\frac{156989}{1244758} c+\frac{1604109}{622379}+\eta}+
            \varepsilon X^{\frac{61084569}{12301745}-c+\eta} \ll \varepsilon X^{\frac{61084569}{12301745}-c+\eta}.
\end{align}
From (\ref{4-inner-integral-fenjie}), (\ref{4-inner-integral-fenjie-1}) and (\ref{4-inner-integral-fenjie-2}), we deduce that
\begin{equation*}
  \int_{\tau<|x|<K}|S(x)|^4|\Phi(x)||\mathcal{T}(x-y)|\mathrm{d}x\ll\varepsilon X^{\frac{61084569}{12301745}-c+\eta},
\end{equation*}
from which and (\ref{S^4-Phi-upper}), we get
\begin{align}\label{S^5-upper-result}
\int_{\tau<|x|<K} |S(x)|^5|\Phi(x)|\mathrm{d}x & \ll X^{\frac{1}{2}}(\log X)\bigg( X^{2+\eta} \cdot \varepsilon X^{\frac{61084569}{12301745}-c+\eta}\bigg)^{\frac{1}{2}}\nonumber\\
& \ll \varepsilon^{\frac{1}{2}}X^{\frac{48994902}{12301745}-\frac{c}{2}+\eta}.
\end{align}

Next, we estimate the integral
\begin{equation*}
\int_{\tau<|x|<K} |S(x)|^6|\Phi(x)|\mathrm{d}x.
\end{equation*}
Similarly, we have
\begin{align*}
\int_{\tau<|x|<K} |S(x)|^6|\Phi(x)|\mathrm{d}x
= & \,\, \sum_{X<p\leqslant2X}(\log p)\int_{\tau<|x|<K}e(p^cx)\overline{S(x)}|S(x)|^4|\Phi(x)|\mathrm{d}x\\
\ll & \,\, (\log X)\sum_{X<n\leqslant2X}\Bigg|\int_{\tau<|x|<K}e(n^cx)\overline{S(x)}|S(x)|^4|\Phi(x)|\mathrm{d}x\Bigg|.
\end{align*}
By Cauchy's inequality, we obtain
\begin{align}\label{S^6-Cauchy-upper}
& \,\, \int_{\tau<|x|<K} |S(x)|^6|\Phi(x)|\mathrm{d}x\nonumber\\
\leqslant & \,\, X^{\frac{1}{2}}(\log X)\Bigg(\sum_{X<n\leqslant2X}\Bigg|\int_{\tau<|x|<K}e(n^cx)\overline{S(x)}|S(x)|^4|\Phi(x)|\mathrm{d}x\Bigg|^2\Bigg)^{\frac{1}{2}}\nonumber\\
= & \,\, X^{\frac{1}{2}}(\log X)\Bigg(\sum_{X<n\leqslant2X}\int_{\tau<|y|<K}\overline{e(n^cy)\overline{S(y)}|S(y)|^4|\Phi(y)|}\mathrm{d}y\nonumber\\
  & \,\, \qquad \qquad\qquad\qquad\times\int_{\tau<|x|<K}e(n^cx)\overline{S(x)}|S(x)|^4|\Phi(x)|\mathrm{d}x\Bigg)^{\frac{1}{2}}\nonumber\\
\leqslant & \,\, X^{\frac{1}{2}}(\log X)\Bigg(\int_{\tau<|y|<K}|S(y)|^5|\Phi(y)|\mathrm{d}y\int_{\tau<|x|<K}|S(x)|^5|\Phi(x)||\mathcal{T}(x-y)|\mathrm{d}x\Bigg)^{\frac{1}{2}}.
\end{align}
For the inner integral in (\ref{S^6-Cauchy-upper}), we get
\begin{align}\label{5-inner-integral-fenjie}
& \,\, \int_{\tau<|x|<K}|S(x)|^5|\Phi(x)||\mathcal{T}(x-y)|\mathrm{d}x\nonumber\\
\ll & \,\, \int_{\substack{\tau<|x|<K \\ |x-y|\leqslant X^{-c}}}|S(x)|^5|\Phi(x)||\mathcal{T}(x-y)|\mathrm{d}x+\int_{\substack{\tau<|x|<K \\ X^{-c}<|x-y|\leqslant 2K}}|S(x)|^5|\Phi(x)||\mathcal{T}(x-y)|\mathrm{d}x.
\end{align}
By the trivial estimate, we have $\mathcal{T}(x-y)\ll X$, which combines with Lemma \ref{xiaobei-lemma} and Lemma \ref{S(x)-yuqujianguji-5} to obtain
\begin{align}\label{5-inner-integral-fenjie-1}
& \,\,\int_{\substack{\tau<|x|<K \\ |x-y|\leqslant X^{-c}}}|S(x)|^5|\Phi(x)||\mathcal{T}(x-y)|\mathrm{d}x \nonumber\\
\ll & \,\, \varepsilon X\times \sup_{\tau<|x|<K} |S(x)|^5\times
              \int_{\substack{\tau<|x|<K \\ |x-y|\leqslant X^{-c}}} \mathrm{d}x \nonumber \\
\ll & \,\, \varepsilon X\cdot X^{\frac{60978530}{12301745}-c}\ll \varepsilon X^{\frac{73280275}{12301745}-c}.
\end{align}
According to (\ref{T(x)-expo-pair-explicit-4}), (\ref{S^5-upper-result}) and Lemma \ref{S(x)-yuqujianguji-5}, we get
\begin{align}\label{5-inner-integral-fenjie-2}
 & \,\, \int_{\substack{\tau<|x|<K \\ X^{-c}<|x-y|\leqslant2K}}|S(x)|^5|\Phi(x)||\mathcal{T}(x-y)|\mathrm{d}x
                     \nonumber \\
 \ll & \,\, \int_{\substack{\tau<|x|<K \\ X^{-c}<|x-y|\leqslant2K}}|S(x)|^5|\Phi(x)|
              \bigg(|x-y|^\frac{156989}{1244758} X^{\frac{156989}{1244758} c+\frac{359351}{622379}}
               +\frac{1}{|x-y|X^{c-1}} \bigg)\mathrm{d}x
                     \nonumber \\
 \ll & \,\, X^{\frac{156989}{1244758} c+\frac{359351}{622379}+\eta}
            \int_{\tau<|x|<K}\!\!|S(x)|^5|\Phi(x)|\mathrm{d}x
                      \nonumber \\
    & \,\,  \qquad+\varepsilon X^{1-c}\times \sup_{\tau<|x|<K}|S(x)|^5\times \int_{\substack{\tau<|x|<K \\ X^{-c}<|x-y|\leqslant2K}}\frac{\mathrm{d}x}{|x-y|}
                        \nonumber \\
  \ll & \,\, X^{\frac{156989}{1244758} c+\frac{359351}{622379}+\eta}\cdot \varepsilon^{\frac{1}{2}}X^{\frac{48994902}{12301745}-\frac{c}{2}+\eta}
             +\varepsilon X^{1-c}\cdot X^{\frac{60978530}{12301745}+\eta}
                        \nonumber \\
 \ll & \,\, \varepsilon^{\frac{1}{2}}X^{\frac{34914042479353}{7656347751355}-\frac{465390}{1244758}c+\eta}+
            \varepsilon X^{\frac{73280275}{12301745}-c+\eta} \ll \varepsilon X^{\frac{73280275}{12301745}-c+\eta}.
\end{align}
From (\ref{5-inner-integral-fenjie}), (\ref{5-inner-integral-fenjie-1}) and (\ref{5-inner-integral-fenjie-2}), we deduce that
\begin{equation*}
  \int_{\tau<|x|<K}|S(x)|^5|\Phi(x)||\mathcal{T}(x-y)|\mathrm{d}x\ll\varepsilon X^{\frac{73280275}{12301745}-c+\eta},
\end{equation*}
from which and (\ref{S^5-upper-result}), we get
\begin{align}\label{S^6-upper-result}
\int_{\tau<|x|<K} |S(x)|^6|\Phi(x)|\mathrm{d}x & \ll X^{\frac{1}{2}}(\log X)\bigg( \varepsilon^{\frac{1}{2}}X^{\frac{48994902}{12301745}-\frac{c}{2}+\eta} \cdot \varepsilon X^{\frac{73280275}{12301745}-c+\eta}\bigg)^{\frac{1}{2}}\nonumber\\
& \ll \varepsilon^{\frac{3}{4}}X^{\frac{134576922}{24603490}-\frac{3}{4}c+\eta}.
\end{align}

Finally define
\begin{align*}
 & B_1(R)  = \int_{-\infty}^{+\infty}S^3(x)e(-Rx)\Phi(x)\mathrm{d}x,\qquad    D_1(R) = \int_{-\tau}^{+\tau}S^3(x)e(-Rx)\Phi(x)\mathrm{d}x,\\
 & D_2(R)  = \int_{\tau<|x|<K}S^3(x)e(-Rx)\Phi(x)\mathrm{d}x,\qquad
  D_3(R) = \int_{|x|\geqslant K}S^3(x)e(-Rx)\Phi(x)\mathrm{d}x,\\
 & H(R) = \int_{-\infty}^{+\infty}I^3(x)e(-Rx)\Phi(x)\mathrm{d}x,\qquad
  H_1(R) = \int_{-\tau}^{+\tau}I^3(x)e(-Rx)\Phi(x)\mathrm{d}x.
\end{align*}
We have
\begin{align}\label{|B_1-H|^2-fenjie}
& \,\,\int_{N}^{2N}|B_1(R)-H(R)|^2\mathrm{d}R \nonumber\\
= & \,\, \int_{N}^{2N}\big|D_1(R)-H_1(R)+D_2(R)+D_3(R)-\big(H(R)-H_1(R)\big)\big|^2\mathrm{d}R \nonumber\\
\ll & \,\, \int_{N}^{2N}|D_1(R)-H_1(R)|^2\mathrm{d}R+\int_{N}^{2N}|D_2(R)|^2\mathrm{d}R\nonumber\\
 &\qquad\qquad\qquad+\int_{N}^{2N}|D_3(R)|^2\mathrm{d}R+\int_{N}^{2N}|H(R)-H_1(R)|^2\mathrm{d}R.
\end{align}
According to (3.18) of Zhang and Li \cite{Zhang-Li-2018}, we have
\begin{equation}\label{|D_1-H_1|^2-upper}
\int_{N}^{2N}|D_1(R)-H_1(R)|^2\mathrm{d}R \ll \varepsilon^2N^{\frac{6}{c}-1}E^{\frac{2}{3}}\mathscr{L}^7.
\end{equation}
Take $\Omega_1=\{R: N<R\leqslant 2N\},\,\Omega_2=\{x: \tau<|x|<K\},\, \xi=S^3(x)\Phi(x),\,\omega(x,R)=e(Rx),\,c(R)=\overline{D_2(R)}$. Then from Lemma \ref{Laporta-lemma-2}, we obtain
\begin{align}\label{|D_2|^2-first}
\int_{N}^{2N}|D_2(R)|^2\mathrm{d}R & = \int_{\Omega_1}\overline{D_2(R)}\big\langle S^3(x)\Phi(x), e(Rx)\big\rangle_2\mathrm{d}R\nonumber\\
& \ll \bigg(\int_{\Omega_2}|S^3(x)\Phi(x)|^2\mathrm{d}x\bigg)^{\frac{1}{2}}
\bigg(\int_{\Omega_1}|\overline{D_2(R)}|^2\mathrm{d}R\bigg)^{\frac{1}{2}}\nonumber\\
& \qquad \qquad \times \bigg(\sup_{x'\in\Omega_1}\int_{\Omega_1}\bigg|\big\langle e(Rx), e(Rx')\big\rangle_2\bigg|\mathrm{d}x\bigg)^{\frac{1}{2}}.
\end{align}
According to (\ref{S^6-upper-result}), (\ref{|D_2|^2-first}) and Lemma \ref{Laporta-lemma-1}, we get
\begin{align}\label{|D_2|^2-upper}
\int_{N}^{2N}|D_2(R)|^2\mathrm{d}R & \ll A\int_{\tau<|x|<K}|S(x)|^6|\Phi(x)|^2\mathrm{d}x  \ll \varepsilon \mathscr{L} \int_{\tau<|x|<K}|S(x)|^6|\Phi(x)|\mathrm{d}x \nonumber\\
& \ll \varepsilon \mathscr{L} \cdot \varepsilon^{\frac{3}{4}}X^{\frac{134576922}{24603490}-\frac{3}{4}c+\eta} \ll \varepsilon^2N^{\frac{6}{c}-1}E^{\frac{2}{3}}\mathscr{L}^7.
\end{align}
By Lemma \ref{xiaobei-lemma}, we get
\begin{align}\label{|D_3|^2-upper}
\int_{N}^{2N}|D_3(R)|^2\mathrm{d}R & \ll \int_{N}^{2N}\Bigg|\int_{K}^{+\infty}|S(x)|^3|\Phi(x)|\mathrm{d}x\Bigg|^2\mathrm{d}R \nonumber\\
& \ll NX^6\Bigg|\int_{K}^{+\infty}\bigg(\frac{5r}{\pi x\varepsilon}\bigg)^r\frac{\mathrm{d}x}{x}\Bigg|^2 \nonumber\\
& \ll NX^6\bigg(\frac{5r}{\pi K\varepsilon}\bigg)^{2r} \ll N.
\end{align}
From Lemma \ref{xiaobei-lemma}, we derive that
\begin{align}\label{|H-H_1|^2-upper}
\int_{N}^{2N}|H(R)-H_1(R)|^2\mathrm{d}R & \ll \int_{N}^{2N}\bigg|\int_{|x|>\tau}|I(x)|^3|\Phi(x)|\mathrm{d}x\bigg|^2\mathrm{d}R \nonumber\\
& \ll \varepsilon^2NX^{6-6c}\bigg|\int_{\tau}^{+\infty}\frac{\mathrm{d}x}{x^3}\bigg|^2 \ll \varepsilon^2N^{\frac{2}{c}-1+\eta},
\end{align}
where we use the estimate
\begin{equation}\label{I(x)-upper}
   I(x)\ll \frac{1}{|x|X^{c-1}},
\end{equation}
which follows from Lemma 4.2 in Titchmarsh \cite{Titchmarsh-book}.

Combining (\ref{|B_1-H|^2-fenjie}), (\ref{|D_1-H_1|^2-upper}), (\ref{|D_2|^2-upper}), (\ref{|D_3|^2-upper}) and (\ref{|H-H_1|^2-upper}), we obtain
\begin{equation}\label{|B_1-H|^2-upper}
\int_{N}^{2N}|B_1(R)-H(R)|^2\mathrm{d}R \ll \varepsilon^2N^{\frac{6}{c}-1}E^{\frac{2}{3}}\mathscr{L}^7.
\end{equation}
This implies that for all $R\in(N,2N]\setminus \mathfrak{A}$ with $|\mathfrak{A}|= O(NE^\frac{2}{15})$, we have
\begin{equation}\label{B_1=H+O}
B_1(R) = H(R) + O\big(\varepsilon N^{\frac{3}{c}-1}E^\frac{1}{10}\big).
\end{equation}
Define
\begin{equation*}
B(R)=\sum_{\substack{X<p_1,p_2,p_3\leqslant 2X\\|p_1^c+p_2^c+p_3^c-R|<\varepsilon}}(\log p_1)(\log p_2)(\log p_3).
\end{equation*}
By the property of $\phi(y)$ in Lemma \ref{xiaobei-lemma} and the Fourier transformation formula, we have
\begin{align}\label{B_1-ll-B}
B_1(R) & =\sum_{X<p_1,p_2,p_3\leqslant 2X}(\log p_1)(\log p_2)(\log p_3)\int_{-\infty}^{+\infty}e\big((p_1^c+p_2^c+p_3^c-R)x\big)\Phi(x)\mathrm{d}x\nonumber\\
 & = \sum_{X<p_1,p_2,p_3\leqslant 2X}(\log p_1)(\log p_2)(\log p_3)\phi(p_1^c+p_2^c+p_3^c-R)\leqslant B(R).
\end{align}
Proceeding as in \cite{Tolev-1992}, Lemma 6, we get
\begin{equation}\label{H-lower}
H(R)\gg \varepsilon R^{\frac{3}{c}-1}.
\end{equation}
Hence, from (\ref{|B_1-H|^2-upper})-(\ref{H-lower}), for all $R\in(N,2N]\setminus \mathfrak{A}$ with $|\mathfrak{A}|= O(NE^\frac{2}{15})$, we have
\begin{equation*}
B(R)\geqslant B_1(R)\geqslant H(R)\gg \varepsilon R^{\frac{3}{c}-1},
\end{equation*}
which finishes the proof of Theorem \ref{Theorem-three-almost all}.

\section{Proof of Theorem \ref{Theorem-five-variables-inequality} }

In this section, we set $X=\frac{1}{2}(\frac{N}{5})^{\frac{1}{c}}$ and denote by $\phi(y)$ and $\Phi(x)$ the functions which appear in Lemma \ref{xiaobei-lemma} with parameter $a=\frac{9\varepsilon}{10},b=\frac{\varepsilon}{10},r=[\log X]$. Define
\begin{equation*}
  \mathscr{B}_6(N)=\sum_{\substack{X<p_1,p_2,p_3,p_4,p_5,p_6\leqslant2X \\ |p_1^c+\cdots+p_6^c-N|<\varepsilon }}(\log p_1)(\log p_2)\cdots(\log p_6).
\end{equation*}
By the property of $\phi(y)$, we have $\mathscr{B}_6(N)\geqslant \mathscr{C}_{6}(N)$, where
\begin{equation*}
  \mathscr{C}_6(N)=\sum_{X<p_1,p_2,p_3,p_4,p_5,p_6\leqslant2X }
  (\log p_1)\cdots(\log p_6)\phi(p_1^c+\cdots+p_6^c-N).
\end{equation*}
From the Fourier transformation formula, we derive that
\begin{align}\label{C_5(N)-fenie}
 \mathscr{C}_6(N)= & \,\, \sum_{X<p_1,\dots,p_6\leqslant2X}(\log p_1)\cdots(\log p_6)
                         \int_{-\infty}^{+\infty}e\big((p_1^c+\cdots+p_6^c-N)y\big)\Phi(y)\mathrm{d}y
                               \nonumber \\
 = & \,\,  \int_{-\infty}^{+\infty}S^6(x)\Phi(x)e(-Nx)\mathrm{d}x
                               \nonumber \\
 =& \,\, \bigg(\int_{|x|\leqslant\tau}+\int_{\tau<|x|<K}+\int_{|x|\geqslant K}\bigg)S^6(x)\Phi(x)e(-Nx)\mathrm{d}x
                                \nonumber \\
 =& \,\, \mathscr{C}_6^{(1)}(N)+\mathscr{C}_6^{(2)}(N)+\mathscr{C}_6^{(3)}(N),\quad \textrm{say}.
\end{align}

Define
\begin{align*}
  & \,\, \mathcal{H}(N)=\int_{-\infty}^{+\infty} I^6(x)\Phi(x)e(-Nx)\mathrm{d}x,  \\
  & \,\, \mathcal{H}_\tau(N)=\int_{-\tau}^{+\tau} I^6(x)\Phi(x)e(-Nx)\mathrm{d}x.
\end{align*}
From (\ref{I(x)-upper}) and Lemma \ref{xiaobei-lemma}, we derive that
\begin{align}\label{H_tau-H-5}
  \big|\mathcal{H}(N)-\mathcal{H}_\tau(N)\big|\ll \int_{\tau}^\infty|I(x)|^6|\Phi(x)|\mathrm{d}x
              \ll \varepsilon\int_{\tau}^{\infty}\bigg(\frac{1}{|x|X^{c-1}}\bigg)^6\mathrm{d}x\ll  \varepsilon X^{6-c-\eta},
\end{align}
By Lemma \ref{xiaobei-lemma},
Lemma \ref{S(x)-I(x)-fourth-power-lemma} and Lemma \ref{S(x)=I(x)+jieyu}, we deduce that
\begin{align}\label{C_5^(1)-H_tau}
   & \,\, \big|\mathscr{C}_6^{(1)}(N)-\mathcal{H}_\tau(N)\big|\leqslant
          \int_{-\tau}^{+\tau}\big|S^6(x)-I^6(x)\big|\big|\Phi(x)\big|\mathrm{d}x
                     \nonumber \\
   \ll & \,\, \varepsilon\cdot\int_{-\tau}^{+\tau}\big|S(x)-I(x)\big|\big(|S(x)|^5+|I(x)|^5\big)\mathrm{d}x
                     \nonumber \\
   \ll & \,\, \varepsilon X\cdot \sup_{|x|\leqslant \tau}\big|S(x)-I(x)\big|
              \bigg(\int_{-\tau}^{+\tau}|S(x)|^4\mathrm{d}x+\int_{-\tau}^{+\tau}|I(x)|^4\mathrm{d}x\bigg)
                     \nonumber \\
   \ll & \,\, \varepsilon X\cdot X\exp\big(-(\log X)^{1/5}\big)
              \bigg(\int_{-\tau}^{+\tau}|S(x)|^4\mathrm{d}x+\int_{-\tau}^{+\tau}|I(x)|^4\mathrm{d}x\bigg)
                     \nonumber \\
   \ll & \,\, \varepsilon X^{6-c}\exp\big(-(\log X)^{1/6}\big).
\end{align}
It follows from Lemma \ref{5-power-main-low-bound}, (\ref{H_tau-H-5}) and (\ref{C_5^(1)-H_tau}) that
\begin{equation}\label{C_5^(1)(N)-lower-bound}
  \mathscr{C}_6^{(1)}(N)=\big( \mathscr{C}_6^{(1)}(N)-\mathcal{H}_\tau(N)\big)+\big(\mathcal{H}_\tau(N)-\mathcal{H}(N)\big)+\mathcal{H}(N)
                       \gg \varepsilon  X^{6-c}.
\end{equation}
By the definition of $\mathscr{C}_6^{(2)}(N)$ and (\ref{S^6-upper-result}), we obtain
\begin{align}\label{C_5^(2)(N)-upper-bound}
  \big|\mathscr{C}_6^{(2)}(N)\big|=& \,\, \Bigg|\int_{\tau<|x|<K}S^6(x)\Phi(x)e(-Nx)\mathrm{d}x\Bigg|
  \ll  \,\, \int_{\tau<|x|<K} |S(x)|^6|\Phi(x)|\mathrm{d}x \nonumber\\
  \ll & \,\,\varepsilon^{\frac{3}{4}}X^{\frac{134576922}{24603490}-\frac{3}{4}c+\eta}
  \ll \,\,\varepsilon X^{6-c-\eta}.
\end{align}
By Lemma \ref{xiaobei-lemma}, we have
\begin{align}\label{C_5^(3)(N)-upper-bound}
        \big|\mathscr{C}_6^{(3)}(N)\big|\ll & \,\, \int_{K}^\infty |S(x)|^{6}|\Phi(x)|\mathrm{d}x
           \ll   X^6\int_{K}^\infty\frac{1}{\pi|x|}\bigg(\frac{r}{2\pi|x|b}\bigg)^r\mathrm{d}x
                     \nonumber \\
        \ll & \,\, X^6\bigg(\frac{r}{2\pi b}\bigg)^r\int_K^\infty\frac{\mathrm{d}x}{x^{r+1}}\ll \frac{X^6}{r}\bigg(\frac{r}{2\pi Kb}\bigg)^r
                      \nonumber \\
        \ll  & \,\,\frac{X^6}{\log X}\cdot\bigg(\frac{1}{2\pi\log^5X}\bigg)^{\log X}
                   \ll \frac{X^6}{X^{5\log\log X+\log(2\pi)}(\log X)}\ll 1.
\end{align}
From (\ref{C_5(N)-fenie}), (\ref{C_5^(1)(N)-lower-bound}), (\ref{C_5^(2)(N)-upper-bound}) and (\ref{C_5^(3)(N)-upper-bound}), we deduce that
\begin{equation*}
\mathscr{C}_6(N)=\mathscr{C}_6^{(1)}(N)+\mathscr{C}_6^{(2)}(N)+\mathscr{C}_6^{(3)}(N)\gg \varepsilon X^{6-c},
\end{equation*}
and thus
\begin{equation*}
\mathscr{B}_6(N)  \geqslant\mathscr{C}_6(N) \gg \varepsilon X^{6-c}\gg \frac{X^{6-c}}{\log^4X},
\end{equation*}
which finishes the proof of Theorem \ref{Theorem-five-variables-inequality}.

\section*{Acknowledgement}

The authors would like to express the most sincere gratitude to the referee for his/her patience in
refereeing this paper. This work is supported by the National Natural Science Foundation
of China (Grant No. 11901566, 11971476), the Fundamental Research Funds for the Central
Universities (Grant No. 2019QS02), and the Scientific Research Funds of Beijing Information Science and Technology
University (Grant No. 2025035).

\end{document}